\documentclass[letter,11pt]{amsart}
\usepackage{amssymb} 
\usepackage{graphicx}
\usepackage{color}
\usepackage[font=footnotesize]{caption}
\usepackage{mathtools}

\usepackage{hyperref}
\usepackage{float}
\usepackage{calc}
\def\thetitle{{ Is a typical bi-Perron algebraic unit a pseudo-Anosov dilatation?}}
\hypersetup{
   pdftitle=   \thetitle,
   pdfauthor=  {Hyungryul Baik, Ahmad Rafiqi, and Chenxi Wu}
}

\newtheorem{thm}{Theorem}
\newtheorem*{thm*}{Theorem}
\newtheorem{lem}[thm]{Lemma}
\newtheorem{cor}[thm]{Corollary}

\theoremstyle{remark}

\usepackage{tikz}
\usetikzlibrary{calc,decorations.markings}
\usetikzlibrary{shapes,snakes}

\theoremstyle{definition}
\newtheorem*{defn*}{Definition}

\newcommand\Z{\mathbb{Z}}

\title\thetitle

\author{Hyungryul Baik}
\address{Mathematisches Institut,
Rheinische Friedrich-Wilhelms-Universit\"{a}t Bonn,
Endenicher Allee 60,
53115 Bonn, Germany}
\email{baik@math.uni-bonn.de}

\author{Ahmad Rafiqi}
\address{Department of Mathematics, Cornell University, Malott Hall, Ithaca, NY 14853, USA}
\email{ar776@cornell.edu}

\author{Chenxi Wu}
\address{Max Planck Institute for Mathematics, Bonn, Germany}
\email{chenxiwu@mpim-bonn.mpg.de}

\begin{document}
\maketitle

\begin{abstract}
In this note, we deduce a partial answer to the question in the title. 
In particular, we show that asymptotically almost all bi-Perron algebraic unit whose characteristic polynomial has degree at most $2n$ 
do not correspond to dilatations of pseudo-Anosov maps on a closed orientable surface of genus $n$ for $n\geq 10$. As an application of the argument, 
we also obtain a statement on the number of closed geodesics of the same length in the moduli space of area one abelian differentials for low genus cases. 
\end{abstract}


Thurston's celebrated theorem \cite{Thurston88} classified self-homeomorphisms of compact orientable surfaces up to isotopy. 
As long as a homeomorphism does not admit a finite power which is isotopic to the identity and
does not admit a finite union of disjoint simple closed curves which are preserved up to isotopy, it is 
isotopic to a so-called \emph{pseudo-Anosov map}. Associated to a pseudo-Anosov map is a positive number 
called its \emph{dilatation}. The logarithm of the dilatation of a pseudo-Anosov map is the topological entropy 
of the map. See \cite{FLP} for details. 

We call a number $\lambda$ to be a \emph{pseudo-Anosov dilatation} if it is the dilatation constant for a pseudo-Ansov map on some compact orientable surface. 
We call an algebraic integer $\lambda$ \emph{bi-Perron} if all the Galois conjugates of $\lambda$ are contained in an annulus of outer radius $\lambda$ itself and inner radius $1/\lambda$. Fried \cite{Fried} showed that every pseudo-Anosov dilatation is a bi-Perron algebraic unit, and he conjectured that this is also a sufficient condition: every bi-Perron algebraic unit is realized as the dilatation constant of some pseudo-Anosov map. 

This turned out to be a difficult question and there are many related works. Thurston \cite{Thurston14} gave an answer to an ${\rm Out}(F_n)$-version of this question. In  \cite{Thurston14}, Thurston also gave a few examples of constructions of pseudo-Anosov maps for given bi-Perron algebraic units 
(here the numbers are given as the slopes of some post-critically finite piecewise-linear self-homeomorphisms of the unit interval). These examples motivated the authors' previous work to construct pseudo-Anosov maps in \cite{BRW16} for a bi-Perron algebraic unit given as the leading eigenvalue of a Perron-Frobenius matrix satisfying some additional properties. Both constructions in \cite{Thurston14} and \cite{BRW16} are similar to the one given in \cite{CH04}.  
See also, for instance, \cite{AY80}, \cite{Long85}, \cite{Penner88}, \cite{Leininger04}, \cite{Hir10}, \cite{LT11a}, \cite{LT11b}, \cite{SS15} for various constructions of pseudo-Anosov maps. 

In this short note, we formulate an easier version of Fried's question, and give a statistical answer to it. To state our result precisely, we introduce some notation. 

Let $n \ge 2$ be fixed, and $R$ be any positive real number.  
Define $\mathcal{B}_n(R)$ to be the set of bi-Perron algebraic units no larger than $R$ whose characteristic polynomial has degree at most $2n$. Here, by the characteristic polynomial of a bi-Perron algebraic unit $\lambda$, we mean the monic palindromic integral polynomial whose leading root is $\lambda$ and has the lowest degree among all such polynomials. And let $\mathcal{D}_n(R)$ be the set of dilatations no larger than $R$ of pseudo-Anosov maps with orientable invariant foliations on a closed orientable surface $S_n$ of genus $n$. Similarly,  let $\mathcal{D'}_n(R)$ be the set of dilatations no larger than $R$ of pseudo-Anosov maps, not necessarily with orientable invariant foliation, on a closed orientable surface with genus $n$.

We remark that $\mathcal{D}_n(R)$ is contained in $\mathcal{B}_n(R)$, and $\mathcal{D'}_n(R)$ is contained in $\mathcal{B}_{3n-3}(R)$.  
A pseudo-Anosov dilatation $\lambda$ on a surface of genus $n$ is a root of an integral palindromic 
polynomial of degree at most $2n$ if its invariant foliations are orientable. This is because $\lambda$ is the leading eigenvalue of the induced symplectic action on the homology group of the surface, which is $\mathbb{Z}^{2n}$. If we do not require the invariant foliation to be orientable, the upper bound on degree is $6n-6$: we can reduce this case to the case of orientable foliation by taking a double cover of the surface, and this bound follows from the fact that a quadratic differential on a surface of genus $n$ has at most $2n$ zeros which is due to Gauss-Bonnet together with the Riemann-Hurwitz formula.  

Note that Fried's conjecture is equivalent to $\mathcal{B}_n(R)$ being contained in $\mathcal{D}_m(R)$ 
for some large enough $m$. But a priori, $m$ could be arbitrarily large and we do not know how to prove or disprove the claim. 
Instead we show the following. 

\begin{thm}
\label{thm:main}
Let $\mathcal{B}_n(R)$, $\mathcal{D'}_n(R)$ and $\mathcal{D}_n(R)$ be as above. Then 
\begin{enumerate}
\item $$\lim_{R\rightarrow\infty}{|\mathcal{D}_m(R) \cap \mathcal{B}_n(R)|\over|\mathcal{B}_n(R)|}=0\ ,$$
where $4m-3 \leq n(n+1)/2$. 
In particular, $\lim_{R\rightarrow\infty}{|\mathcal{D}_n(R)|\over|\mathcal{B}_n(R)|}=0, \forall n \geq 6$.
\item $$\lim_{R\rightarrow\infty}{|\mathcal{D'}_m(R) \cap \mathcal{B}_n(R)|\over|\mathcal{B}_n(R)|}=0\ ,$$
where $6m-6 \leq n(n+1)/2$. 
In particular, $\lim_{R\rightarrow\infty}{|\mathcal{D'}_n(R) \cap \mathcal{B}_n(R)|\over|\mathcal{B}_n(R)|}=0, \forall n \geq 10$.
\end{enumerate}
\end{thm}
Here $|A|$ means the cardinality of $A$ for a finite set $A$.

Theorem \ref{thm:main} says that asymptotically almost all bi-Perron algebraic units whose characteristic polynomial has degree at most $2n$ do not correspond to 
dilatations of pseudo-Anosov maps on a surface of genus $n$ for all $n \geq 10$, and for $n\geq 6$ if the invariant foliation is further assumed to be orientable. 
It would be interesting to see if the statement still holds for lower genera. 

Let $\Gamma_n$ be the set of all periodic orbits for the Teichm\"uller flow on the moduli space of area one abelian differentials on $S_n$.
Then there exists a surjective map $\Gamma_n \to \cup_{R > 0} \mathcal{D}_n(R)$ defined by $\gamma \mapsto e^{\ell(\gamma)}$
where $\ell(\gamma)$ is the length of the orbit $\gamma \in \Gamma_n$. 
Let $\Gamma_n(R)$ be the preimage of $\mathcal{D}_n(R)$ of this map. 

By $f\sim g$ we mean $\exists C$ such that ${1\over C}f(x)\leq g(x)\leq Cf(x)$ when $x\gg 0$. By $f\lesssim g$ we mean $f=O(g)$ when $x\rightarrow\infty$. At our best knowledge, the following theorem is independently due to Eskin-Mirzakhani-Rafi \cite{EMR12} and Hamenst\"adt \cite{Hamen13b} (which is rephrased for our purpose).  
\begin{thm}[\cite{EMR12}, \cite{Hamen13b}]
\label{thm:EMRH}
$|\Gamma_n(R)| \sim \frac{R^{4n-3}}{(4n-3) \log R}$.
\end{thm} 

We give a brief explanation of the above statement. 
In \cite{EMR12} and \cite{Hamen13b}, 
it was stated that when restricted to each connected component of the strata of 
the moduli space of area one abelian differentials on $S_n$, 

\begin{align}  
|\Gamma_n(R)| &= |\{\gamma \in \Gamma_n : e^{\ell(\gamma)} \leq R\}| \notag \\
&= |\{\gamma \in \Gamma_n : \ell(\gamma) \leq \log R\}| \notag \\
&\sim \frac{e^{(2n + \ell -1) \log R }}{(2n + \ell -1) \log R} = \frac{R^{2n+\ell - 1}}{(2n+\ell - 1) \log R}. \notag
\end{align}

Here $\ell$ is the maximum number of zeros of an area one abelian differential on $S_n$, so it is $2n-2$. 
See \cite{Masur82}, \cite{Veech86}, \cite{EM11} and \cite{EMR12} for the relevant background of this theorem. 

As a result, we have $|\Gamma_n(R)| \sim \frac{R^{4n-3}}{(4n-3) \log R}$ on each connected component of the strata. 
But for fixed $n$, there exists only finite number of such components. 
Therefore we get $|\Gamma_n(R)| \sim \frac{R^{4n-3}}{(4n-3) \log R}$ without restricting to the components. 
Note that since $n$ is fixed, we can just say $|\Gamma_n(R)| \sim \frac{R^{4n-3}}{\log R}$.

As a direct corollary, we have:
\begin{cor} 
\label{cor:EMRH}
$|\mathcal{D}_n(R)| \lesssim \frac{R^{4n-3}}{\log R}$.
\end{cor} 

In the exactly same way, we can obtain an analogue of Corollary \ref{cor:EMRH} for $|\mathcal{D'}_n(R)|$ from the 
following theorem which is due to Eskin and Mirzakhani \cite{EM11}:
\begin{thm}[Theorem 1.1, \cite{EM11}]
The number of geodesics in the moduli space of genus $n$ surface of length at most $\log(R)$ $\sim {R^{6n-6}\over (6n-6)\log(R)}$.
\end{thm}

And as a direct corollary, just like Corollary \ref{cor:EMRH}, we have:
\begin{cor} 
\label{cor:em}
$|\mathcal{D'}_n(R)| \lesssim \frac{R^{6n-6}}{\log R}$.
\end{cor}

We remark that this does not imply $|\mathcal{D}_n(R)| \sim \frac{R^{4n-3}}{\log R}$ or $|\mathcal{D'}_n(R)| \sim \frac{R^{6n-6}}{\log R}$, since each element in $\mathcal{D}_n(R)$ or $\mathcal{D'}_n(R)$ may correspond to a lot of different closed geodesics in the moduli space. 

Now we study how $|\mathcal{B}_n(R)|$ grows. Let $P_n(R)$ be the set of Perron polynomials of degree $n$ with roots no larger than $R$, ($x>1$ is Perron if it is the root of a monic irreducible polynomial with integer coefficients, so that the other roots of the polynomial are (strictly) less than $x$ in absolute value). 

\begin{lem}
\label{lem:Pn}
	$|P_n(R)|\sim R^{n(n+1)/2}$.
\end{lem}
\begin{proof}
$|P_n(R)|\lesssim R^{n(n+1)/2}$: Due to Vieta's formula the absolute value of the coefficient of $x^k$ of a monic, degree $n$ polynomial with all roots no larger than $R$ is $\lesssim R^{n-k}$. Hence, the total number of such polynomials must be $\lesssim \prod_k R^{n-k}=R^{n(n+1)/2}$.\\

$R^{n(n+1)/2}\lesssim |P_n(R)|$: Let $a_k$ be the coefficient of $x^k$ in a degree $n$ monic polynimial (so $a_n=1$). By Rouch\'e's theorem, when
\begin{equation}\label{1}
1>|{a_0\over R^n}|+|{a_1\over R^{n-1}}|+\dots + |{a_{n-1}\over R}|
\end{equation}
\begin{equation}\label{2}
{\left({1\over 2}\right)}^{n-1}|{a_{n-1}\over R}|>|{a_0\over R^n}|+\dots+|{a_{n-2}\over R^2}|{\left({1\over 2}\right)}^{n-2} + {\left({1\over 2}\right)}^n
\end{equation}
and
\begin{equation}\label{3}
{\left({1\over 3}\right)}^{n-1}|{a_{n-1}\over R}|>|{a_0\over R^n}|+\dots+|{a_{n-2}\over R^2}|{\left({1\over 3}\right)}^{n-2} + {\left({1\over 3}\right)}^n
\end{equation}
one root $\lambda$ of this polynomial has magnitude between $R/2$ and $R$ while all other roots have magnitude smaller than $R/3$. Hence $\lambda$ must be real. If $\lambda<0$ one can multiply $(-1)^{n-k}$ to $a_k$ to get a polynomial with a root $-\lambda$. Hence,  
 half of those polynomials have a leading positive real root, so they are in $P_n(R)$. 
 
The inequalities  (\ref{1}), (\ref{2}) and  (\ref{3}) are satisfied if and only if the point \newline $(a_0/R^n, \ldots, a_{n-1}/R)$ lies in a non-empty open subset $U\subset[-1,1]^n$. The number of such points as $R\rightarrow\infty$ converges to the volume of this open subset divided by 
the co-volume of the lattice $\Z/R^n \times \Z/R^{n-1} \times \ldots \times \Z/R$, and the co-volume of this lattice is $R^{-n(n+1)/2}$. \\
 \end{proof}

\begin{lem}
\label{lem:reduciblePn}
   $\lim_{R\rightarrow\infty} {|\{\text{reducible elements in }P_n(R)\}|\over|P_n(R)|}=0$
\end{lem}

\begin{proof}
	Because any reducible monic integer polynomial can be written as the product of two monic integer polynomials of lower degree, we have:
	$$|\{\text{reducible elements in }P_n(R)\}|\leq \sum_k |b_k(R)||b_{n-k}(R)|$$
	where $b_k(R)$ is the set of monic polynomials with roots bounded by $R$. The first part of the proof of the previous lemma implies that $|b_k(R)|\lesssim R^{k(k+1)/2}$, hence $|\{\text{reducible elements in }P_n(R)\}|\sim o(R^{n(n+1)/2})$.
\end{proof}

Let $\mathcal{P}_n(R)$ be the set of Perron numbers of degree $n$ no larger than $R$. They have one-one correspondence with irreducible elements in $P_n(R)$ which by Lemma \ref{lem:reduciblePn} constitute almost all of $P_n(R)$ asymptotically. Hence by Lemma \ref{lem:Pn}, $|\mathcal{P}_n(R)|\sim R^{n(n+1)/2}$.

\begin{lem}
\label{lem:Bn}
	 $|\mathcal{B}_n(R)|\sim R^{n(n+1)/2}$.
\end{lem}
\begin{proof}
	When $x>1$, $1/x<1$. Hence, $x\in \mathcal{B}_n(R)$ implies that $x+1/x\in \mathcal{P}_n(R+1)$, hence $|\mathcal{B}_n(R)|\lesssim (R+1)^{n(n+1)/2}\sim R^{n(n+1)/2}$.
	
	On the other hand, from the proof of Lemma \ref{lem:Pn} and the discussion above, the number of Perron numbers of degree $n$ which lie between $R$ and $R/2$ and have all conjugates smaller than $R/3$ is $\sim R^{n(n+1)/2}$. When $R$ is sufficiently large, each such Perron number $y$ correspond to a bi-Perron number $x$ by the relation $x+1/x=y$. And in fact $x$ is an algebraic unit (so it is in $\mathcal{B}_n(R)$). Note that both roots of the polynomial $x^2 - yx + 1$ are algebraic integers, since it is a monic polynomial with algebraic integer coefficients. Furthermore, the product of these roots is 1, so they must be algebraic units. 
Hence $\mathcal{B}_n(R)\gtrsim R^{n(n+1)/2}$.
\end{proof}

Now we are ready to prove our main theorem. 
\begin{proof}[Proof of Theorem \ref{thm:main}] 
By Corollary \ref{cor:EMRH} and Lemma \ref{lem:Bn}, we get 
$$ {|\mathcal{D}_m(R) \cap \mathcal{B}_n(R)|\over|\mathcal{B}_n(R)|} \lesssim \frac{R^{4m-3}}{\log R \cdot R^{n(n+1)/2}}.$$ The right-hand side goes to $0$ as $R$ goes to $+\infty$ as long as we have $4m-3 \leq n(n+1)/2$. When $m = n$, this inequality is satisfied if and only if $n \geq 6$ (recall that $n$ is always assumed to be at least $2$), which proved part (1). Part (2) follows from the same argument above but using Corollary \ref{cor:em} instead of Corollary \ref{cor:EMRH}. 
\end{proof} 

Recall that $\Gamma_n$ is the set of all closed geodesics in the moduli space of area one abelian differentials on the surface $S_n$, 
and $\Gamma_n(R)$ is the subset of $\Gamma_n$ which consists of the closed geodesics of length no larger than $\log R$. For each 
$\gamma \in \Gamma_n$, let $m_\gamma$ be the number $| \{ \gamma' \in \Gamma_n : \ell(\gamma') = \ell(\gamma) \}|$. 

We remark that if $m_\gamma$ were uniformly bounded, one could have obtained Theorem \ref{thm:main} (1) for $n \geq 2$ instead of $n \geq 6$ using 
Theorem 1 (1) of \cite{Hamen16}. But at least in the low genus cases, this is not true. As an application of our argument, we obtain the following theorem. 

\begin{thm} \label{thm:multiplicity} 
Suppose $n \leq 5$. 
For any positive integer $k$, the set $\{ \gamma \in \Gamma_n : m_\gamma \geq k \}$ is typical, i.e., 
$$\lim_{R\to \infty} \dfrac{| \{ \gamma \in \Gamma_n(R) : m_\gamma \geq k \} |}{| \Gamma_n(R) |} \to 1.$$

\end{thm} 
\begin{proof}
Suppose not. Then for some $k$, we have $$\limsup \dfrac{| \{ \gamma \in \Gamma_n(R) : m_\gamma < k \} |}{| \Gamma_n(R) |} > 0.$$

But this implies that 
$$ \limsup \dfrac{ |\mathcal{D}_n(R)| }{ |\Gamma_n(R) |} \geq \limsup \dfrac{\frac{1}{k} | \{ \gamma \in \Gamma_n(R) : m_\gamma < k \} |}{| \Gamma_n(R) |} > 0.$$

On the other hand, we know that $\mathcal{D}_n(R) \subset \mathcal{B}_n(R)$. As a consequence, 
$ \lim \dfrac{ |\mathcal{B}_n(R)| }{ |\Gamma_n(R)|} \geq  \limsup \dfrac{ |\mathcal{D}_n(R)| }{ |\Gamma_n(R) |}$. 
By Corollary \ref{cor:EMRH} and Lemma \ref{lem:Bn}, we get 
$$ \dfrac{ |\mathcal{B}_n(R)| }{ |\Gamma_n(R)|} \sim \dfrac{ R^{n(n+1)/2} \log R }{ R^{4n-3} } \to 0, \mbox{ for } n \leq 5, $$ 
a contradiction.

\end{proof}

\section{Acknowledgements}

\noindent We would like to thank Bal\'azs Strenner for many helpful comments. The first author was  supported by the ERC Grant Nb. 10160104. \\[10pt]


\begin{thebibliography}{9}
\bibitem{AY80} Arnoux, P., Yoccoz, J.(1980). \emph{Construction de diff\'{e}omorphisms pseudo-Anosov}, C. R. Acad. Sci. Paris 292, 75--78. \\


\bibitem{BRW16} Baik, H., Rafiqi, A. and Wu, C. (2016). \emph{Constructing pseudo-Anosov maps with given dilatations}, Geom. Dedicata, Volume 180, Issue 1, 39--48. \\

\bibitem{CH04} de Carvalho, A. and Hall, T.  (2004). \emph{Unimodal generalized pseudo-Anosov maps.} Geom. Topol., 8, 1127--1188. \\

\bibitem{EM11}  Eskin, A., Mirzakhani, M. (2011). \emph{Counting closed geodesics in moduli space}, J. Mod. Dynamics 5, 71--105.\\

\bibitem{EMR12} Eskin, A., Mirzakhani, M., Rafi, K. \emph{Counting closed geodesics in strata}, arXiv:1206.5574. \\


\bibitem{FMbook} Farb, B., Margalit, D. (2012). \emph{A primer on mapping class groups}, volume 49 of Princeton Mathematical Series. Princeton University Press, Princeton, NJ. \\ 

\bibitem{Fried}  Fried, D. (1985). \emph{Growth rate of surface homeomorphisms and flow equivalence.} Ergod. Theory Dyn. Syst., 5(04), 539--563.\\

\bibitem{Hamen13} Hamenst\"adt, U. \emph{Symbolic dynamics for the Teichm\"uller flow}, arXiv:1112.6107.\\

\bibitem{Hamen13b} Hamenst\"adt, U. (2013). \emph{Bowen's construction for the Teichm\"uller flow}, J. Mod. Dynamics 7, 489--526.\\

\bibitem{Hamen16}  Hamenst\"adt, U. \emph{Typical properties of periodic Teichm\"uller geodesics}, arXiv:1409.5978.\\

\bibitem{Hir10} Hironaka, E. (2010). \emph{Small dilatation mapping classes coming from the simplest hyperbolic braid}, Algebr. Geom. Topol., 10(4):2041--2060. \\


\bibitem{LT11a} Lanneau, E., Thiffeault, J. (2011). \emph{On the minimum dilatation of braids on punctured discs}, Geom. Dedicata, 152:165--182. \\
\bibitem{LT11b} Lanneau, E., Thiffeault, J. (2011). \emph{On the minimum dilatation of pseudoAnosov homeromorphisms on surfaces of small genus}, Ann. Inst. Fourier (Grenoble), 61(1):105--144. \\

\bibitem{Leininger04} Leininger, C. (2004) \emph{On groups generated by two positive multi-twists: Teichmuller curves and Lehmer's number}, Geom. Topol., 8:1301--1359.\\

\bibitem{Long85} Long, D. (1985). \emph{Constructing pseudo-Anosov maps}, In Knot theory and manifolds (Vancouver, B.C., 1983), volume 1144 of Lecture Notes in Math., 108--114. Springer, Berlin.\\

\bibitem{Masur82}  Masur, H. (1982). \emph{Interval exchange transformations and measured foliations}, Ann. Math. 115, 169--201. \\


\bibitem{Penner88} Penner, R. (1988). \emph{A construction of pseudo-Anosov homeomorphisms}, Trans. Amer.Math. Soc., 310(1):179--197.\\

\bibitem{SS15} Shin, H., Strenner, B. (2015). \emph{Pseudo-Anosov mapping classes not arising from Penner's construction}, Geom. Topol. 19, 3645--3656. \\ 

\bibitem{Thurston88} Thurston, W. (1988). \emph{On the geometry and dynamics of diffeomorphisms of surfaces}, Bull. Am. Math. Soc. (New Ser.) 19, 417--431.\\

\bibitem{Thurston14} Thurston, W. (2014). \emph{Entropy in dimension one.} arXiv:1402.2008.\\

\bibitem{FLP} \emph{Travaux de Thurston sur les surfaces}, volume 66 of Asterisqu\'{e} . Soci\'{e}t\'{e} Math\'{e}matique de France, Paris, 1979. S\'{e}minaire Orsay, With an English summary. \\

\bibitem{Veech86}  Veech, W. (1986). \emph{The Teichm\"uller geodesic flow}, Ann. Math. 124, 441--530.\\

\end{thebibliography}
\end{document}